\documentclass[reqno]{amsart}

\usepackage[sorting=none, sorting=nyt, maxbibnames=99, backend=biber]{biblatex}

\renewbibmacro{in:}{}
\bibliography{references.bib}
\usepackage{amssymb}
\usepackage{calrsfs}
\usepackage{graphics,graphicx, mathrsfs}
\usepackage{enumerate}
\usepackage{enumitem}
\usepackage{url}
\usepackage{xcolor}
\definecolor{vio}{rgb}{0.54, 0.17, 0.89}
\usepackage{bbold}
\usepackage[ruled, lined, linesnumbered, longend]{algorithm2e}
\usepackage{hyperref}
\usepackage[justification=centering]{caption}
\renewcommand{\thefootnote}{\alph{footnote}}

\newtheorem{theorem}{Theorem}[section]
\newtheorem{lemma}[theorem]{Lemma}

\newtheorem{corollary}[theorem]{Corollary}
\numberwithin{equation}{section}

\theoremstyle{remark}
\newtheorem*{remark}{Remark}

\DeclareMathOperator{\R}{\mathbb{R}}
\DeclareMathOperator{\N}{\mathbb{N}}

\def\reals{\hbox{\rm I\kern-.18em R}}
\def\complexes{\hbox{\rm C\kern-.43em
\vrule depth 0ex height 1.4ex width .05em\kern.41em}}
\def\field{\hbox{\rm I\kern-.18em F}} %symbol for field

\let\svthefootnote\thefootnote
\newcommand\freefootnote[1]{%
  \let\thefootnote\relax%
  \footnotetext{#1}%
  \let\thefootnote\svthefootnote%
}

\allowdisplaybreaks

\begin{document}

\title[The error term in the explicit formula of Riemann--von Mangoldt]{On the error term in the explicit formula of Riemann--von Mangoldt II}

\author{Michaela Cully-Hugill and Daniel R. Johnston}
\address{School of Science, UNSW Canberra, Australia}
\email{m.cully-hugill@unsw.edu.au}
\address{School of Science, UNSW Canberra, Australia}
\email{daniel.johnston@unsw.edu.au}
\date\today
% \subjclass[2000]{Primary: 11Y40; Secondary: 11M06, 11M26, 11B65, 11B68}
\keywords{}

\begin{abstract}
    We give an explicit $O(x/T)$ error term for the truncated Riemann--von Mangoldt explicit formula. For large $x$, this provides a modest improvement over previous work, which we demonstrate via an application to a result on primes between consecutive powers. 
\end{abstract}

\maketitle
\freefootnote{\textit{Affiliation}: School of Science, The University of New South Wales Canberra, Australia}
\freefootnote{\textit{Corresponding author}: Daniel Johnston (daniel.johnston@unsw.edu.au)}
\freefootnote{\textit{Other author}: Michaela Cully-Hugill (m.cully-hugill@unsw.edu.au)}
\freefootnote{\textit{Key phrases}: Riemann-von Mangoldt formula, zero-free regions, primes between powers, explicit estimates, Riemann zeta-function}
\freefootnote{\textit{MSC classes}: 11M26 (Primary) 11N05 (Secondary)}

\section{Introduction}

This paper gives a new explicit estimate for the error in the truncated Riemann--von Mangoldt explicit formula. This is a sequel to \cite{CH_DJ_Perron1}, which gave a related result. In contrast to \cite{CH_DJ_Perron1}, this new estimate is asymptotically sharper, but allows for less control over the point of truncation.

Let $\psi(x)$ denote the Chebyshev prime-counting function with
\begin{equation*}
    \psi(x)=\sum_{n\leq x}\Lambda(n)=\sum_{p^k\leq x}\log p,
\end{equation*}
where $p^k$ are prime powers for any integer $k\geq 1$, and $\Lambda(n)$ is the von Mangoldt function. The truncated Riemann--von Mangoldt formula can be written as
\begin{equation}\label{riemannvoneq}
    \psi(x)=x-\sum_{\substack{\rho=\beta+i\gamma\\|\gamma|\leq T}}\frac{x^\rho}{\rho}+E(x,T),
\end{equation}
where the sum is over all non-trivial zeros $\rho=\beta+i\gamma$ of the Riemann zeta-function $\zeta(s)$ that have $|\gamma|\leq T$, and $E(x,T)$ is an error term. 

In \cite[Thm~1.2]{CH_DJ_Perron1}, the authors prove the following.
\begin{theorem}\label{oldRvM}
    For any $\alpha\in(0,1/2]$ there exist constants $M$ and $x_M$ such that for $\max\{51,\log x\}<T<(x^{\alpha}-2)/2$,
    \begin{equation}\label{oldRvMpsi}
        \psi(x)=x-\sum_{\substack{|\gamma|\leq T}}\frac{x^\rho}{\rho}+O^*\left(M\frac{x\log x}{T}\right)
    \end{equation}
    for all $x\geq x_M$. Some admissible values of $x_M$, $\alpha$ and $M$ are $(40,1/2,5.03)$ and $(10^3, 1/100, 0.5597)$, with more given in \cite[Table 4]{CH_DJ_Perron1}.
\end{theorem}
This result is reached via the Perron formula, which expresses the partial sums of an arithmetic function as an integral over its corresponding Dirichlet series. In particular, the truncated Perron formula can be written as (e.g., see \cite[Thm.~5.2, Cor.~5.3]{montgomery2007multiplicative} or \cite[Chp.~4.4]{Murty_2013})
\begin{equation}\label{Perron}
    \sum_{n\leq x} a_n = \frac{1}{2\pi i} \int_{\kappa-iT}^{\kappa+iT} F(s)\frac{x^s}{s} \mathrm{d}s + R(x,T).
\end{equation}
Here, $R(x,T)$ is an error term, and the formula is valid for $x>0$, any Dirichlet series $F(s) = \sum_{n=1}^\infty a_n n^{-s}$ that converges in $\Re(s)>\kappa_a$, any $\kappa>\max(\kappa_a,0)$, and bounded $T>1$. The admissible range for the point of truncation $T$ is usually a function of $x$, and can be somewhat adjusted for explicit estimates.

In Section 3 of \cite{Wolke_1983}, Wolke gave a method to estimate the integral in \eqref{Perron}. This was made explicit in \cite[Thm.~4]{CH_DJ_Perron1} for the case $F(s)=\sum_{n\geq 1}\Lambda(n) n^{-s} = -(\zeta'/\zeta)(s)$ and $\kappa=1+1/\log x$. The authors then adapted a method of Ramar{\'e} in \cite[Thm.~1.1, Thm.~2.1]{Ramare_16_Perron} to estimate $R(x,T)$, \cite[Thm.~2.1]{CH_DJ_Perron1}. Combining these two estimates in (\ref{Perron}) resulted in Theorem \ref{oldRvM}. 

In \cite[Thm.~1.2, Thm.~3.1]{Ramare_16_Perron}, Ramar{\'e} then proved a weighted and averaged version of \eqref{Perron} with an asymptotically sharper estimate for the error term. The purpose of this paper is to make this method of Ramar{\'e} explicit, and apply it in the context of the Riemann--von Mangoldt formula. Although \cite[Thm.~3.1]{Ramare_16_Perron} is already explicit, there are several errors in the proofs and statements of Theorems 1.2 and 3.1 of \cite{Ramare_16_Perron} that we need to correct or circumvent. For instance, there is a missing factor of $k+1$ in the error term of \cite[Thm.~3.1]{Ramare_16_Perron}, a typo in the definition (3.1) of \cite{Ramare_16_Perron}, and some measure-theoretic problems which were confirmed by Ramar{\'e} through private correspondence (see also \cite{ramareaddendum}). After mending these errors, and simplifying the result for our purposes, our main theorem is as follows. 

\begin{theorem}\label{mainthm}
    For any $\alpha\in(0,1/2]$ and $\omega\in[0,1]$ there exist constants $M$ and $x_M$ such that for some $T^*\in [T,2T]$ with $\max\{51,\log^2 x\}<T<(x^{\alpha}-2)/4$, we have
    \begin{equation}\label{mainthmpsi}
        \psi(x)=x-\sum_{\substack{|\gamma|\leq T^*}}\frac{x^\rho}{\rho}+O^*\left(M\frac{x(\log x)^{1-\omega}}{T}\right)
    \end{equation}
    for all $x\geq x_M$. In Table \ref{maintable}, some admissible values of $x_M$ and $M$ are given for specific choices of $\alpha$.
\end{theorem}

\def\arraystretch{1.5}
\begin{table}[h]
\centering
\caption{Some corresponding values of $x_M$, $\alpha$, $\lambda$ and $M$ for Theorem~\ref{mainthm}. Here $\lambda$ is a parameter which we optimise over in the proof of the theorem. The last two entries are specifically used to prove Theorem \ref{powerthm}}
\begin{tabular}{|c|c|c|c|c|}
\hline
$\log(x_M)$ & $\alpha$ & $\omega$ & $\lambda$ & $M$\\
\hline
$40$ & $1/2$ & $0$ & $0.43$ & $2.894$\\
\hline
$10^3$ & $1/2$ & $0$ & $0.42$ & $1.681$\\
\hline
$10^{10}$ & $1/2$ & $0.3$ & $0.42$ & $2.275$\\
\hline
$10^{13}$ & $1/2$ & $1$ & $10^{-4}$ & $19.81$\\
\hline
$10^3$ & $1/10$ & $0.2$ & $0.07$ & $1.260$\\
\hline
$10^{10}$ & $1/10$ & $0.9$ & $10^{-3}$ & $4.415$\\
\hline
$10^3$ & $1/100$ & $0.8$ & $0.07$ & $3.615$\\
\hline
$10^{10}$ & $1/100$ & $1$ & $10^{-3}$ & $9.631$\\
\hline
$10^{3}$ & $1/85$ & $0.9$ & $0.07$ & $6.391$\\
\hline
$4\cdot 10^{3}$ & $1/85$ & $0.9$ & $0.05$ & $5.462$\\
\hline
\end{tabular}
\label{maintable}
\end{table}

Although we now have less control over the point of truncation, for most applications Theorem \ref{mainthm} implies that the factor of $\log x$ in \eqref{oldRvMpsi} can be reduced or removed. At first glance, it may seem preferable to just take $\omega = 1$. However, the corresponding $M$ may be such that the overall error estimate is larger for a wide range of $x$, compared to the $M$ one obtains with a smaller $\omega$. Hence, for small to moderate values of $x$ the best estimate will often come from using some $\omega<1$.

It would also be possible to sharpen the explicit result: we see two immediate avenues. In the proof of Theorem \ref{mainthm}, one could set $\kappa$ to be another function such as $\kappa = 1+c/\log x$ and optimise over $c$. This would also require reworking \cite[Thm.~4]{CH_DJ_Perron1} for such a choice of $\kappa$. Another potential improvement is to modify the range for $T^*$ (currently $[T,2T]$). This would amount to optimising the choice of the parameter $\xi>1$ which we introduce in \S\ref{sectRam}. Either of these improvements would be rather systematic to pursue, but we have chosen not to as they would complicate and lengthen this paper for only a small to moderate improvement. In particular, neither of these improvements would improve the asymptotic form of \eqref{mainthmpsi}. Moreover, the main error terms in the proof of Theorem \ref{mainthm} rely more on other results, such as the size of the zero-free region of $\zeta(s)$.

We also have the following variant of Theorem~\ref{mainthm} for primes in intervals.
\begin{theorem}\label{mainthm-intervals}
    Let $h=h(x)$ be any nonnegative function. For any $\alpha\in(0,1/2]$ and $\omega\in[0,1]$ there exist constants $M$ and $x_M$ such that for some $T^*\in [T,2T]$ with $\max\{51,\log^2 x\}<T<(x^{\alpha}-2)/4$,
    \begin{align*}
        \psi(x+h)-\psi(x)=h-&\sum_{|\gamma|\leq T^*}\frac{(x+h)^\rho-x^\rho}{\rho}\\
        &\quad+O^*\left(M\frac{x+h}{T}(\log(x+h))^{1-\omega}+M\frac{x}{T}(\log x)^{1-\omega}\right).
    \end{align*}
    for all $x\geq x_M$. In Table \ref{maintable}, some admissible values of $x_M$ and $M$ are given for specific choices of $\alpha$.
\end{theorem}

To demonstrate the power of our new estimates we use Theorem~\ref{mainthm-intervals} to obtain the following result.
\begin{theorem}\label{powerthm}
    There is at least one prime between $n^{90}$ and $(n+1)^{90}$ for all $n\geq 1$.
\end{theorem}
This improves upon the previous result in \cite[Thm.~1.3]{CH_DJ_Perron1}, which found primes between consecutive $140^{\text{th}}$ powers. Here, the main improvement comes from our refined error term in Theorem \ref{mainthm-intervals}. However, we have also incorporated some recent zero-free regions \cite{bellotti2024explicit,mossinghoff2024explicit,yang2024explicit} to ensure that Theorem \ref{powerthm} is state-of-the-art.

\subsection{Outline of paper}
In Section \ref{sectRam} we make explicit a weighted averaged version of the truncated Perron formula. From this, we deduce a corollary in Section \ref{sectCor} for a ``modified'' version of the truncated Perron formula for general arithmetic functions. This corollary is then used in conjunction with \cite[Thm.~4.1]{CH_DJ_Perron1} to obtain Theorems~\ref{mainthm} and \ref{mainthm-intervals} in Section \ref{sectMain}. Finally, the proof of Theorem~\ref{powerthm} is given in Section \ref{sectPowers} and other supplementary results are given in the appendices \ref{appzf} and \ref{appk}.

We remark that in Sections \ref{sectCor}, \ref{sectMain}, \ref{sectPowers} and Appendix \ref{appk} some computations are required. Further details about these computations can be found at
\begin{equation}\label{codeeq}\tag{\textbf{CODE}}
    \href{https://github.com/DJmath1729/Riemann-von-Mangoldt-estimates}{github.com/DJmath1729/Riemann-von-Mangoldt-estimates}    
\end{equation}
which is linked with the arXiv version of this paper.

\section{An averaged, weighted, truncated Perron formula}\label{sectRam}
In this section we prove Theorem \ref{mainramthm} stated below, which is an error estimate for a weighted averaged version of the truncated Perron formula. This result is a modification of \cite[Thm. 3.1]{Ramare_16_Perron} and its proof also fixes several minor errors appearing in \cite{Ramare_16_Perron}. 

First, some notation: let $k\in\N$, $\xi\geq 1$. We call a function $w$, defined over $[1,\xi]$, $(k,\xi)$-admissible if it has the following properties:
\begin{itemize}
    \item[(1)] it is $k$-times differentiable, and $w^{(k)}$ is in $L^1$,
    \item[(2)] $\int_1^\xi w(t)\mathrm{d}t=1$,
    \item[(3)] $w^{(\ell)}(1)=w^{(\ell)}(\xi)=0$ for $0\leq\ell\leq k-2$. This condition is empty when $k=1$. 
\end{itemize}
For such a function $w$, we define 
\begin{align}
    N_{k,\xi}(w)&=\frac{1}{2\pi}\left(\frac{1}{\xi}|w^{(k-1)}(\xi)|+|w^{(k-1)}(1)|+k!\sum_{0\leq h\leq k}\frac{1}{h!}\int_1^\xi\frac{|w^{(h)}(u)|}{u}\mathrm{d}u\right),\label{Ndef}\\
    \mathscr{L}_{\xi}(w)&=\frac{1}{\pi}\int_{1}^\xi u|w(u)|\mathrm{d}u\label{Ldef}
\end{align}
and $\theta'=\theta'_{k,\xi}(w)$ is the positive point of intersection of the two functions
\begin{equation}\label{f1f2def}
    f_1(y)=\frac{2N_{k,\xi}(w)}{|y|^{k+1}}\quad\text{and}\quad f_2(y)=1+|y|\mathscr{L}_{\xi}(w).
\end{equation}
Note that such a point of intersection exists since $N_{k,\xi}(w)$ and $\mathscr{L}_{\xi}(w)$ are positive. Moreover, we remark that $\theta'$ is absent in Ramar\'e's analogous argument in \cite{Ramare_16_Perron}. Rather, on page 122 of \cite{Ramare_16_Perron}, Ramar\'e defines a variable $T'=T/\theta$, and the optimum value of $\theta$ is given by $\theta'$ in our notation. Our aim is to prove the following theorem.
\begin{theorem}\label{mainramthm}
    Let $k\geq 1$ be an integer and let $\xi>1$ be a real number. Let $w$ be a $(k,\xi)$-admissible function. Let $F(s)=\sum_{n\geq 1}a_n/n^s$ be a Dirichlet series that converges absolutely for $\Re(s)>\kappa_a$. For $x\geq 1$, $T\geq 1$, and $\kappa>\kappa_a>0$ we have
    \begin{align}\label{lemmaeq}
        \sum_{n\leq x}a_n &= \frac{1}{2\pi i}\int_{T}^{\xi T}\int_{\kappa-it}^{\kappa+it}F(s)\frac{x^s}{s} \mathrm{d}s \frac{w(t/T)}{T} \mathrm{d}t\notag\\
        &+O^*\left((k+1)\frac{2 N_{k,\xi}(w)}{T^{k+1} }\int_{\theta'/T}^\infty\sum_{|\log(x/n)|\leq u}\frac{|a_n|}{n^{\kappa}}\frac{x^{\kappa}\mathrm{d}u}{u^{k+2}}\right).
    \end{align}
\end{theorem}

To prove Theorem \ref{mainramthm} we need the following lemma that uses the function $$v(y) = \begin{cases} 1 \quad y\geq 0 \\ 0 \quad y<0,\end{cases}$$ 
or more concisely $v(y)=\mathbb{1}_{[0,\infty)}(y)$, where $\mathbb{1}_{S}(y)$ is the indicator function of $S\subseteq\R$.

\begin{lemma}\label{averagelem}
    Let $w$ be a $(k,\xi)$-admissible function. For $\kappa'>0$ and $y\in \mathbb{R}$, 
    \begin{equation*}
        \left| v(y) - \frac{1}{2\pi i} \int_1^\xi\int_{\kappa'-i\tau}^{\kappa'+i\tau} \frac{e^{ys}}{s} \mathrm{d}sw(\tau)\mathrm{d}\tau\right|\leq \Delta(y),
    \end{equation*}
    where
    \begin{align}\label{eq:averagelem}
        \Delta(y)=
        \begin{cases}
            \min\left\lbrace \frac{2 e^{y\kappa'}}{|y|^{k+1}}N_{k,\xi}(w), \left| v(y) -\frac{e^{y\kappa'}}{\pi} \arctan(1/\kappa') \right| + |y|e^{y\kappa'}\mathscr{L}_\xi(w) \right\rbrace,&y\neq 0,\\
            \left| v(y) -\frac{e^{y\kappa'}}{\pi} \arctan(1/\kappa') \right| + |y|e^{y\kappa'}\mathscr{L}_\xi(w),&y=0.
        \end{cases}
    \end{align}
\end{lemma}

\begin{remark}
    Lemma \ref{averagelem} is very similar to Ramar\'e's \cite[Lemma 3.2]{Ramare_16_Perron}. However, due to some small errors in Ramar\'e's work, we still include a full proof here, For example, compared to \eqref{eq:averagelem}, in \cite[Lemma 3.2]{Ramare_16_Perron} the factor of 2 in front of $N_{k,\xi}(w)$ is missing, and there is an incorrect factor of $1/\pi$ in front of $\mathscr{L}_{\xi}(w)$.
\end{remark}

\begin{proof}
We begin by noting that
\begin{equation*}
    \int_{\kappa'-i\tau}^{\kappa'+i\tau} \frac{e^{ys}}{s}\mathrm{d}s = e^{y\kappa'} \int_{\kappa'-i\tau}^{\kappa'+i\tau} \frac{1}{s}\mathrm{d}s + i e^{y\kappa'}\int_{-\tau}^\tau \frac{e^{iyt}-1}{\kappa'+it}\mathrm{d}t.
\end{equation*}
The first integral is equal to $2i\text{arctan}(1/\kappa')$. For $y\neq 0$, the second integral can be bounded by $2\tau |y|$ using the identity
\begin{equation*}
    \left| \frac{e^{iyt}-1}{iyt} \right| = \left| \int_{0}^1 e^{iytu} \mathrm{d}u \right| \leq 1,
\end{equation*}
and note that the bound $|e^{iyt}-1|\leq |yt|$ can be separately verified for $y=0$. In particular,
\begin{equation*}
    \left|\int_{\kappa'-i\tau}^{\kappa'+i\tau} \frac{e^{ys}}{s}\mathrm{d}s-\frac{e^{y\kappa'}}{\pi}\arctan(1/\kappa')\right|\leq 2\tau|y|e^{y\kappa'},
\end{equation*}
so that, upon using $\int_1^\xi w(t)\mathrm{d}t=1$,
\begin{align*}
    \left| v(y) - \frac{1}{2\pi i} \int_1^\xi\int_{\kappa'-i\tau}^{\kappa'+i\tau} \frac{e^{sy}}{s} \mathrm{d}sw(\tau)\mathrm{d}\tau\right| &\leq \left| v(y) - \frac{e^{y\kappa'}}{\pi} \text{arctan}(1/\kappa') \right| + |y|e^{y\kappa'}\mathscr{L}_{\xi}(w).
\end{align*}
for all $y\in\R$.

To prove the other part of the $y\neq 0$ case in \eqref{averagelem}, we first consider $y<0$. For any $K>\kappa'$ we have
\begin{equation}\label{fourints}
    \left( \int_{\kappa'-i\tau}^{\kappa'+i\tau} + \int_{\kappa'+i\tau}^{K+i\tau} + \int_{K+i\tau}^{K-i\tau} + \int_{K-i\tau}^{\kappa'-i\tau} \right) \frac{e^{ys}}{s}\mathrm{d}s = 0.
\end{equation}
We multiply both sides of \eqref{fourints} by $w(\tau)$ and integrate over $\tau$ from 1 to $\xi$. The third integral can be bounded by
\begin{align}\label{3rdIntegral}
    \left| \int_1^\xi \int_{K+i\tau}^{K-i \tau} \frac{e^{ys}}{s}\mathrm{d}s w(\tau)\mathrm{d}\tau \right| &= \int_1^\xi \int_{\tau}^{-\tau} \frac{ie^{y(K+it)}}{K+it}\mathrm{d}t  w(\tau)\mathrm{d}\tau \notag\\
    &\leq \int_1^{\xi}\int_{-\tau}^{\tau} \frac{e^{yK}}{K}\mathrm{d}t \left|w(\tau)\right| \mathrm{d}\tau \leq \max_{[1,\xi]}|w(\tau)|\cdot\frac{\xi^2 e^{yK}}{K},
\end{align}
which goes to zero as $K\to \infty$, noting that $\max_{[1,\xi]}|w(\tau)|$ exists because $w(\tau)$ is differentiable and thus continuous on $[1,\xi]$. We now bound the integral from $\kappa'+i\tau$ to $K+i\tau$ in the limit $K\to \infty$, and note that the following argument also gives the same bound for the integral from $K-i\tau$ to $\kappa'-i\tau$. We start by writing
\begin{equation}\label{intswitch}
    \int_1^\xi\int_{\kappa'+i\tau}^{K+i\tau}\frac{e^{sy}}{s}\mathrm{d}sw(\tau)\mathrm{d}\tau=\int_{\kappa'}^Ke^{uy}\int_1^\xi\frac{e^{i\tau y}w(\tau)}{u+i\tau}\mathrm{d}\tau\mathrm{d}u.
\end{equation}
To deal with the inner integral, we set $f(\tau)=w(\tau)/(u+i\tau)$ and apply the Leibnitz formula for the derivative of a product. In particular, for any $m\geq 0$ we have
\begin{equation*}
    f^{(m)}(\tau)=\sum_{0\leq h\leq m}\binom{m}{h}\frac{i^{m-h}(m-h)!w^{(h)}(\tau)}{(u+i\tau)^{m-h+1}}.
\end{equation*}
Since $w^{(\ell)}(1)=w^{(\ell)}(\xi)=0$ for $0\leq\ell\leq k-2$, this implies
\begin{equation*}
    f^{(k-1)}(1)=\frac{w^{(k-1)}(1)}{u+i},\qquad f^{(k-1)}(\xi)=\frac{w^{(k-1)}(\xi)}{u+i\xi}.
\end{equation*}
Integrating by parts $k$ times thus gives
\begin{align*}
    \int_1^\xi\frac{e^{i\tau y}w(\tau)}{u+i\tau}\mathrm{d}\tau&=\frac{(-1)^{(k-1)}f^{(k-1)}(\xi)e^{\xi i y}}{(iy)^k}-\frac{(-1)^{k-1}f^{(k-1)}(1)e^{iy}}{(iy)^k}\\
    &\qquad +\frac{(-1)^{k}}{(iy)^k}\int_1^\xi e^{i\tau y}f^{(k)}(\tau)\mathrm{d}\tau.
\end{align*}
Substituting this into \eqref{intswitch} gives
\begin{align}
    &\int_1^\xi\int_{\kappa'+i\tau}^{K+i\tau}\frac{e^{sy}}{s}\mathrm{d}sw(\tau)\mathrm{d}\tau=\notag\\
    &\qquad\frac{(-1)^{k-1}w^{(k-1)}(\xi)}{(iy)^k}\int_{\kappa'}^K\frac{e^{(u+\xi i)y}}{u+i\xi}\mathrm{d}u-\frac{(-1)^{k-1}w^{(k-1)}(1)}{(iy)^k}\int_{\kappa'}^K\frac{e^{(u+i)y}}{u+i}\mathrm{d}u\notag\\
    &\qquad+\sum_{0\leq h\leq k}\binom{k}{h}\int_{1}^\xi\frac{(-1)^k i^{k-h}(k-h)!w^{(h)}(\tau)}{(iy)^k}\int_{\kappa'}^K\frac{e^{(u+i\tau)y}}{(u+i\tau)^{k-h+1}}\mathrm{d}u\mathrm{d}\tau.\label{threeint}
\end{align}
For the first integral, we bound $1/(u+i\xi)$ by $1/\xi$ and integrate $e^{uy}$ to obtain
\begin{equation*}
    \left|\frac{(-1)^{k-1}w^{(k-1)}(\xi)}{(iy)^k}\int_{\kappa'}^K\frac{e^{(u+\xi i)y}}{u+i\xi}\mathrm{d}u\right|\leq\frac{|w^{(k-1)}(\xi)|}{|y|^k}\cdot\frac{e^{\kappa'y}}{\xi|y|}
\end{equation*}
since $K\to \infty$. The other integrals in \eqref{threeint} are bounded in analogous ways to give
\begin{equation*}
    \int_1^\xi\int_{\kappa'+i\tau}^{K+i\tau}\frac{e^{sy}}{s}\mathrm{d}sw(\tau)\mathrm{d}\tau\leq \frac{2\pi e^{\kappa' y}N_{k,\xi}(w)}{|y|^{k+1}}.
\end{equation*}
% \begin{align}
%     \int_1^\xi\int_{\kappa'+i\tau}^{K+i\tau}\frac{e^{sy}}{s}\mathrm{d}sw(\tau)\mathrm{d}\tau \leq \frac{e^{\kappa'y}}{|y|^{k+1}}\left(\frac{|w^{(k-1)}(\xi)|}{\xi} + |w^{(k-1)}(1)| + k!\sum_{0\leq h\leq k}\int_{1}^\xi \frac{|w^{(h)}(\tau)|}{h!\tau^{k-h+1}} \mathrm{d}\tau \right).
% \end{align}
Using this bound in \eqref{fourints} gives
\begin{equation*}
    \left| v(y) - \frac{1}{2\pi i} \int_{1}^\xi\int_{\kappa'-i\tau}^{\kappa'+i\tau} \frac{e^{ys}}{s}\mathrm{d}sw(\tau)\mathrm{d}\tau \right| \leq \frac{2e^{y\kappa'}}{|y|^{k+1}}N_{k,\xi}(w).
\end{equation*}

The case $y>0$ is almost identical, but for $K<0$ we begin with
\begin{equation}\label{fourints2}
    \left( \int_{\kappa'-i\tau}^{\kappa'+i\tau} + \int_{\kappa'+i\tau}^{K+i\tau} + \int_{K+i\tau}^{K-i\tau} + \int_{K-i\tau}^{\kappa'-i\tau} \right) \frac{e^{ys}}{s}\mathrm{d}s = 2\pi i.
\end{equation}
After multiplying both sides of \eqref{fourints2} by $w(\tau)$ and integrating over $\tau$ from 1 to $\xi$, we find that the third integral approaches 0 as $K\to -\infty$, as it has the same bound as in (\ref{3rdIntegral}). For the integrals from $\kappa'+i\tau$ to $K+i\tau$ and from $K-i\tau$ to $\kappa'-i\tau$, we have the same bounds as before in the limit $K\to -\infty$.
Using these estimates in \eqref{fourints2} gives
\begin{equation*}
    \left| 2\pi i - \int_{1}^\xi\int_{\kappa'-i\tau}^{\kappa'+i\tau} \frac{e^{ys}}{s}\mathrm{d}sw(\tau)\mathrm{d}\tau \right| \leq \frac{4\pi e^{y\kappa'}}{y^{k+1}}N_{k,\xi}(w),
\end{equation*}
and thus
\begin{equation*}
    \left| v(y) - \frac{1}{2\pi i}\int_{1}^\xi\int_{\kappa'-i\tau}^{\kappa'+i\tau} \frac{e^{ys}}{s}\mathrm{d}sw(\tau)\mathrm{d}\tau \right| \leq \frac{2 e^{y\kappa'}}{y^{k+1}}N_{k,\xi}(w).\qedhere
\end{equation*}
\end{proof}

\begin{proof}[{Proof of Theorem \ref{mainramthm}}]
We now prove Theorem \ref{mainramthm}. To do this, we need to bound
\begin{equation*}
    \left| \sum_{n\geq 1} a_n v(T\log(x/n)) - \frac{1}{2\pi i T}\int_{T}^{\xi T}\int_{\kappa-it}^{\kappa+it}F(s)\frac{x^s\mathrm{d}s}{s}w(t/T)\mathrm{d}t \right| x^{-\kappa}
\end{equation*}
for any $x\geq 1$, $T>0$, and $\kappa>\kappa_a>0$, where $\kappa_a$ is the abscissa of absolute convergence of $F(s)$. First, we take $\kappa = \kappa' T$ for any $\kappa'>0$, and use the substitutions $s=\tilde{s}T$ and $t = \tau T$ to obtain
\begin{align}\label{fullsumeq}
     &\left| \sum_{n\geq 1} a_n v(T\log(x/n)) - \frac{1}{2\pi i T}\int_{T}^{\xi T}\int_{\kappa-it}^{\kappa+it}F(s)\frac{x^s\mathrm{d}s}{s}w(t/T)\mathrm{d}t \right| x^{-\kappa}\notag\\
     &\qquad\leq \sum_{n\geq 1} \frac{|a_n|}{n^\kappa} \left| v(T\log(x/n)) - \frac{1}{2\pi i T} \int_{T}^{\xi T}\int_{\kappa-it}^{\kappa+it} \left(\frac{x}{n}\right)^{s}\frac{\mathrm{d}s }{s} w(t/T)\mathrm{d}t \right| \left(\frac{n}{x}\right)^{\kappa} \notag \\
     &\qquad= \sum_{n\geq 1} \frac{|a_n|}{n^\kappa} \left| v(T\log(x/n)) - \frac{1}{2\pi i T} \int_{1}^{\xi}\int_{\kappa'-i\tau}^{\kappa'+i\tau} e^{T\log\left(\frac{x}{n}\right)\tilde{s}}\frac{\mathrm{d}\tilde{s}}{\tilde{s}}w(\tau)\mathrm{d}\tau \right| e^{-\kappa' T\log(x/n)}.
\end{align}
Next, we apply Lemma \ref{averagelem}. For $\theta> 0$ we let
\begin{align}
    c_{\theta}=\max_{0\leq |y| < \theta} \Delta(y)  \label{c22},
\end{align}
where $\Delta(y)$ is as defined in \eqref{eq:averagelem}. We then set $y=T\log(x/n)$ in Lemma \ref{averagelem} and split the sum in \eqref{fullsumeq} at the value of $n$ such that $T|\log(x/n)|=\theta$. This gives
\begin{align*}
    &\left| \sum_{n\geq 1} a_n v(T\log(x/n)) - \frac{1}{2\pi i T}\int_{T}^{\xi T}\int_{\kappa-it}^{\kappa+it}F(s)\frac{x^s\mathrm{d}s}{s}w(t/T)\mathrm{d}t \right| x^{-\kappa}\\
    &\qquad\leq c_{\theta}\sum_{T|\log(x/n)|< \theta} \frac{|a_n|}{n^\kappa} + \frac{2 N_{k,\xi}(w)}{T^{k+1}} \sum_{T|\log(x/n)| \geq \theta} \frac{|a_n|}{n^\kappa |\log(x/n)|^{k+1}}.
\end{align*}
Now, since
\begin{align*}
    &\sum_{T|\log(x/n)| \geq \theta} \frac{|a_n|}{n^\kappa |\log(x/n)|^{k+1}} \\
    &\qquad=\sum_{T|\log(x/n)| \geq \theta}\frac{|a_n|}{n^{\kappa}}\int_{|\log(x/n)|}^\infty\frac{(k+1)}{u^{k+2}}\mathrm{d}u\\
    &\qquad=(k+1)\int_{\theta/T}^\infty\sum_{\theta/T\leq|\log(x/n)|\leq u}\frac{|a_n|}{n^{\kappa}}\frac{\mathrm{d}u}{u^{k+2}}\\
    &\qquad=(k+1)\int_{\theta/T}^\infty\sum_{|\log(x/n)|\leq u}\frac{|a_n|}{n^{\kappa}}\frac{\mathrm{d}u}{u^{k+2}}-\frac{T^{k+1}}{\theta^{k+1}}\sum_{|\log(x/n)|< \theta/T}\frac{|a_n|}{n^{\kappa}}
\end{align*}
we have
\begin{align}\label{finalceq}
    &\left| \sum_{n\leq x} a_n v(T\log(x/n)) - \frac{1}{2\pi i T}\int_{T}^{\xi T}\int_{\kappa-it}^{\kappa+it}F(s)\frac{x^s\mathrm{d}s}{s}w(t/T)\mathrm{d}t \right| x^{-\kappa}\notag\\
    &\qquad\leq(k+1)\frac{2 N_{k,\xi}(w)}{T^{k+1} }\int_{\theta/T}^\infty\sum_{|\log(x/n)|\leq u}\frac{|a_n|}{n^{\kappa}}\frac{\mathrm{d}u}{u^{k+2}}\notag\\
    &\qquad\qquad\qquad\qquad\qquad\qquad\qquad+\left(c_{\theta}-\frac{2N_{k,\xi}(w)}{\theta^{k+1}}\right)\sum_{|\log(x/n)|<\theta/ T}\frac{|a_n|}{n^{\kappa}}.
\end{align}

Note that $\left|v(y)e^{-y\kappa'}-\frac{1}{\pi} \arctan(1/\kappa') \right|<1$, so taking $\theta=\theta'$ we have
\begin{equation*}
    c_{\theta'}<\frac{2N_{k,\xi}(w)}{(\theta')^{k+1}}
\end{equation*}
and the second term in \eqref{finalceq} is negative. This proves the theorem.
\end{proof}

\section{A modified Perron formula}\label{sectCor}

As Theorem \ref{mainramthm} involves an integral from $T$ to $\xi T$, we can set $\xi=2$ and use a continuity argument to deduce a strong bound for some $T^*\in [T,2T]$. First, we need to choose a weight function $w(t)$. Roughly following \cite[Sect.~5]{Ramare_16_Perron}, we define
\begin{equation}\label{wdef}
    w(u)=6(u-1)(2-u).
\end{equation}
for $u\in[1,2]$. Note that
\begin{equation}\label{eulerbetaeq}
    \int_1^2 w(u)\mathrm{d}u=1
\end{equation}
and $w(u)$ is a $(1,2)$-admissible function. With this choice of $w(u)$, we compute the functions $\mathscr{L}_2(w)$, $N_{1,2}(w)$ and $\theta'_{1,2}(w)$ defined in Section \ref{sectRam}. 
\begin{lemma}
    Let $w(u)$ be as defined in \eqref{wdef}, and $\mathscr{L}_2(w)$, $N_{1,2}(w)$ and $\theta'_{1,2}(w)$ be as defined in \eqref{Ndef}, \eqref{Ldef} and the ensuing text. Then,
    \begin{equation*}
        \mathscr{L}_{2}(w)=\frac{3}{2\pi},
    \end{equation*}
    \begin{equation}\label{nest}
        N_{1,2}(w)=\frac{9-18\left(\log\frac{4}{3}-\log\frac{3}{2}\right)-12\log 2}{2\pi}\leq 0.4461.
    \end{equation}
    and
    \begin{equation*}
        0.8029\leq\theta'_{1,2}(w)\leq 0.8030.
    \end{equation*}
\end{lemma}
\begin{proof}
    To begin with, $\mathscr{L}_2(w)$ is computed as an elementary integral
    \begin{equation}\label{Lint}
        \mathscr{L}_2(w)=\frac{6}{\pi}\int_1^2 u(u-1)(2-u)\mathrm{d}u=\frac{3}{2\pi}.
    \end{equation}
    To compute $N_{1,2}(w)$, we note that $w(1)=w(2)=0$ so that 
    \begin{equation}\label{Nint}
        N_{1,2}(w)=\frac{1}{2\pi}\left(\int_{1}^2\frac{|w(u)|}{u}+\int_{1}^2\frac{|w'(u)|}{u}\right),
    \end{equation}
    which evaluates to \eqref{nest}. Finally, to compute $\theta'_{1,2}(w)$ we use the \textbf{NSolve} function in \emph{Mathematica} with the functions $f_1(y)$ and $f_2(y)$ defined in \eqref{f1f2def}. Full computational details, including the evaluation of \eqref{Lint} and \eqref{Nint}, are included in the supplementary code for this paper \eqref{codeeq}.
\end{proof}
We now substitute this choice of $w(u)$ into Theorem \ref{mainramthm}, giving the following result.

\begin{corollary}\label{newCor-MTPerron}
    Let $w$ and $\theta'=\theta'_{1,2}(w)$ be as above. Let $x\geq 1$ and $T>1$. Let $F(s)=\sum_{n\geq 1}a_n n^{-s}$ be a Dirichlet series and $\kappa>0$ be a real parameter larger than $\kappa_a$ --- the abscissa of absolute convergence of $F$. For any $\lambda\geq\theta'/T$, there exists a $T^*\in[T,2T]$ such that
    \begin{equation}\label{rameq}
        \sum_{n\leq x}a_n=\frac{1}{2\pi i}\int_{\kappa-iT^*}^{\kappa+iT^*}F(s)\frac{x^s}{s}\mathrm{d}s+O^*(\mathfrak{R}),
    \end{equation}
    where $\mathfrak{R}=\mathfrak{R}(x,T,\kappa)$ satisfies
    \begin{equation}\label{req1}
       \mathfrak{R}\leq \frac{1.785}{T^2} \left(\frac{x^{\kappa}}{2\lambda^{2}}\sum_{n\geq 1}\frac{|a_n|}{n^{\kappa}} + e^{\kappa\lambda}\int_{\theta'/T}^{\lambda} \sum_{|\log(x/n)| \leq u} \frac{|a_n| \mathrm{d}u}{u^3}\right).
    \end{equation}
\end{corollary}
\begin{proof}

Let $\xi=2$ and $k=1$ in Theorem \ref{mainramthm}, and set
\begin{equation}\label{req2}
    \mathfrak{R}'=\frac{4 N_{1,2}(w)}{T^2 }\int_{\theta'/T}^\infty\sum_{|\log(x/n)|\leq u}\frac{|a_n|}{n^{\kappa}}\frac{x^{\kappa}\mathrm{d}u}{u^{3}}
\end{equation}
to be the error term in \eqref{lemmaeq}. A change of variables $u=t/T$ in \eqref{eulerbetaeq} gives 
\begin{equation}\label{eulerbeta2}
    \int_{T}^{2T}\frac{w(t/T)}{T}\mathrm{d}t=1.
\end{equation}
Thus, \eqref{lemmaeq} can be rewritten as
\begin{equation}\label{pirates}
    \int_T^{2 T}\left(\sum_{n\leq x}a_n-\frac{1}{2 \pi i}\int_{\kappa-it}^{\kappa+it}F(s)\frac{x^s}{s}\mathrm{d}s\right)\frac{w(t/T)}{T}\mathrm{d}t=O^*(\mathfrak{R}'),
\end{equation}
where $\mathfrak{R}'$ is defined in \eqref{req2}.
Let $\epsilon=1-1/2\pi$ and 
\begin{equation}
    G(x,t) = \sum_{n\leq x}a_n-\frac{1}{2 \pi i}\int_{\kappa-it}^{\kappa+it} F(s)\frac{x^s}{s}\mathrm{d}s.
\end{equation}

Given (\ref{eulerbeta2}) and that $G(x,t)$ is continuous in $t$, (\ref{pirates}) implies there exists a $T^*\in[T,2T]$ with
\begin{equation}\label{sequenceeq}
    \left| G(x,T^*) \right| \leq  \mathfrak{R}'.
\end{equation}

In particular, if $G(x,t)>\mathfrak{R}'$ for all $t\in[T,2T]$ then
\begin{equation*}
    \int_T^{2T}G(x,t)\frac{w(t/T)}{T}\mathrm{d}t>\mathfrak{R}'\int_T^{2T}\frac{w(t/T)}{T}\mathrm{d}t=\mathfrak{R}',
\end{equation*}
which contradicts \eqref{pirates}.

Next we note that
\begin{equation*}
    \int_{\lambda}^\infty\sum_{|\log(x/n)|\leq u}\frac{|a_n|}{n^{\kappa}}\frac{x^{\kappa}\mathrm{d}u}{u^{3}}\leq x^{\kappa}\sum_{n\geq 1}\frac{|a_n|}{n^{\kappa}}\int_{\lambda}^\infty\frac{1}{u^3}\mathrm{d}u=\frac{x^{\kappa}}{2\lambda^2}\sum_{n\geq 1}\frac{|a_n|}{n^{\kappa}}.
\end{equation*}
Hence, splitting the integral in \eqref{req2} at $u=\lambda$ and using $N_{1,2}(w)\leq 0.4461$ gives
    \begin{align}\label{rineqs}
        |G(x,T^*)|\leq\mathfrak{R}' &\leq \frac{4\cdot 0.4461}{T^2} \left(\frac{x^{\kappa}}{2\lambda^2}\sum_{n\geq 1}\frac{|a_n|}{n^{\kappa}} + \int_{\theta'/T}^{\lambda}\sum_{|\log(x/n)|\leq u}|a_n|\left(\frac{x}{n}\right)^{\kappa}\frac{1}{u^{3}}\mathrm{d}u\right)\notag\\
        &\leq \frac{1.785}{T^2} \left(\frac{x^{\kappa}}{2\lambda^{2}}\sum_{n\geq 1}\frac{|a_n|}{n^{\kappa}} + e^{\kappa\lambda}\int_{\theta'/T}^{\lambda} \sum_{|\log(x/n)| \leq u} \frac{|a_n|}{u^3}\mathrm{d}u\right)=\mathfrak{R},
    \end{align}
    as desired.
\end{proof}

We also have the following variant of Corollary~\ref{newCor-MTPerron} for sums $\sum_{x<n\leq x+h}a_n$.
\begin{corollary}\label{intervalperron}
    Keep the notation of Corollary~\ref{newCor-MTPerron}. For any nonnegative function $h=h(x)$, we have
    \begin{align}\label{aninterval}
        \sum_{x<n\leq x+h}a_n=\frac{1}{2\pi i}\int_{\kappa-iT^*}^{\kappa+iT^*}F(s)\frac{(x+h)^s-x^s}{s}\mathrm{d}s + O^*(\mathfrak{R}(x+h,T,\kappa)+\mathfrak{R}(x,T,\kappa)).
    \end{align}
\end{corollary}
\begin{proof}
    We repeat the argument in the proof of Corollary~\ref{newCor-MTPerron} for
    \begin{equation*}
        \sum_{x<n\leq x+h}a_n=\sum_{n\leq x+h}a_n-\sum_{n\leq x}a_n.
    \end{equation*}
    In particular, by Theorem       \ref{mainramthm} we have (c.f. \eqref{pirates})
    \begin{align}\label{pirates2}
        &\int_T^{2 T}\left(\sum_{x<n\leq x+h}a_n-\frac{1}{2 \pi i}\int_{\kappa-it}^{\kappa+it}F(s)\frac{(x+h)^s-x^s}{s}\mathrm{d}s\right)\frac{w(t/T)}{T}\mathrm{d}t\notag\\
        &\qquad\qquad\qquad\qquad\qquad\qquad\qquad\qquad\qquad\qquad=O^*(\mathfrak{R}'(x+h)+\mathfrak{R}'(x)),
    \end{align}
    where $\mathfrak{R}'(x)=\mathfrak{R'}$ is as defined in \eqref{req2}. Then, as in the proof of Corollary \ref{newCor-MTPerron}, we can set
    \begin{equation*}
        G(x,t)=\sum_{x<n\leq x+h}a_n-\frac{1}{2 \pi i}\int_{\kappa-it}^{\kappa+it}F(s)\frac{(x+h)^s-x^s}{s}\mathrm{d}s
    \end{equation*}
    and by \eqref{pirates2} there exists $T^*\in[T,2T]$ such that 
    \begin{equation*}
        |G(x,T^*)|\leq \mathfrak{R}'(x+h)+\mathfrak{R}'(x)\leq \mathfrak{R}(x+h,T,\kappa)+\mathfrak{R}(x,T,\kappa),
    \end{equation*}
    as required.
\end{proof}

Note that simply applying Corollary~\ref{newCor-MTPerron} separately for 
\begin{equation*}
    \sum_{n\leq x+h}a_n\quad\text{and}\quad \sum_{n\leq x}a_n
\end{equation*}
is not sufficient to prove Corollary~\ref{intervalperron}. This is because we require $T^*\in[T,2T]$ to be the same value in both the integrals
\begin{equation*}
    \int_{\kappa-iT^*}^{\kappa+iT^*}F(s)\frac{(x+h)^s}{s}\mathrm{d}s\qquad\text{and}\qquad\int_{\kappa-iT^*}^{\kappa+iT^*}F(s)\frac{(-x^s)}{s}\mathrm{d}s.
\end{equation*}

\section{Proof of theorems \ref{mainthm} and \ref{mainthm-intervals}}\label{sectMain}

To prove Theorem \ref{mainthm} we use Corollary \ref{newCor-MTPerron} for $a_n=\Lambda(n)$ and $\kappa=1+1/\log x$. Namely, for $x\geq x_M\geq e^{40}$ and $\max\{51,\log^2 x\}<T<\frac{x^{\alpha}-2}{4}$ with any $\alpha\in(0,1/2]$ we have that for some $T^* \in[T,2T]$
\begin{align}\label{perronbound1}
        &\psi(x) = \frac{1}{2\pi i} \int_{\kappa-iT^*}^{\kappa+iT^*} \left(-\frac{\zeta'}{\zeta}(s)\right) \frac{x^s}{s} \mathrm{d}s \notag\\
        &\qquad\qquad\qquad +O^{*}\left(\frac{1.785}{T^2} \left(\frac{x^{\kappa}}{2\lambda^{2}}\sum_{n\geq 1}\frac{\Lambda(n)}{n^{\kappa}} + e^{\kappa\lambda}\int_{\theta'/T}^{\lambda} \sum_{|\log(x/n)| \leq u} \frac{\Lambda(n)}{u^3}\right)\mathrm{d}u\right),
\end{align}
where $\lambda$ is a free parameter and $\theta'$ is as defined in Corollary \ref{newCor-MTPerron}. 

For the first term, Theorem \ref{wolkeprop} in Appendix \ref{appk} implies that for any $\omega\in[0,1]$ there exists a constant $K$ corresponding to $x_M$ (given as $x_K$ in Table \ref{ktable}) such that
\begin{equation}
    \frac{1}{2\pi i}\int_{\kappa-iT^*}^{\kappa+iT^*}\left(-\frac{\zeta'}{\zeta}(s)\right)\frac{x^s}{s}\mathrm{d}s=x-\sum_{|\gamma|\leq T^*}\frac{x^{\rho}}{\rho}+O^*\left(K\frac{x}{T}(\log x)^{1-\omega}\right).
\end{equation}
Next, we use that $x^{\kappa}=ex$ and (see e.g.\ \cite{Delange_87})
\begin{equation*}
    \sum_{n\geq 1}\frac{\Lambda(n)}{n^{\kappa}}=-\frac{\zeta'}{\zeta}(\kappa)<\frac{1}{1-\kappa}=\log x
\end{equation*}
to bound the first sum in \eqref{perronbound1} by
\begin{equation}\label{firstsumbound}
    1.785\cdot\frac{x^\kappa}{2\lambda^2T^2} \sum_{n\geq 1} \frac{\Lambda(n)}{n^\kappa}<1.785\cdot\frac{ex\log x}{2\lambda^2 T^2}.
\end{equation}
For the second term in \eqref{perronbound1}, the condition $|\log(x/n)|\leq u$ is equivalent to $xe^{-u}\leq n\leq xe^u$. For all $u\in[0,\lambda]$ we have $e^{-u}\geq 1-u$ and $e^u\leq 1+(c_0-1)u$ with $$c_0 = c_0(\lambda) = \frac{e^\lambda-1}{\lambda}+1.$$ Hence
\begin{align}\label{perronbound2}
    \frac{1.785 e^{\kappa\lambda}}{T^2}\int_{\theta'/T}^{\lambda} \sum_{|\log(x/n)| \leq u}  \frac{\Lambda(n)}{u^3}\mathrm{d}u&\leq \frac{1.785 e^{\kappa\lambda}}{T^2}\int^{\lambda}_{\theta'/T}\frac{1}{u^3}\sum_{I(x,u)}\Lambda(n)\mathrm{d}u,
\end{align}
where
\begin{equation*}
    I(x,u)=\left\{n\geq 1\::\:x-ux\leq n\leq x + (c_0-1)ux \right\}.
\end{equation*}
Let $x_-=\max\{x-ux,0\}$ and $x_+ = x+(c_0-1)ux$ for $\theta'/T\leq u\leq \lambda$. Defining 
\begin{equation*}
    \theta(x)=\sum_{p\leq x}\log p,
\end{equation*}
we have by an explicit form of the Brun--Titchmarsh theorem \cite[Thm. 2]{montgomery1973large},
\begin{align*}
    \theta(x_+)-\theta(x_-)&\leq \frac{2 \log x_+}{\log(x_+-x_-)}(x_+-x_-) \\
    &\leq \frac{2 \log( x+(c_0-1)ux)}{\log(c_0 ux)}c_0ux \\
    &\leq c_0 ux\cdot\mathcal{E}_1(x,T),
\end{align*}
where
\begin{equation*}
    \mathcal{E}_1(x,T)=\frac{2 \log( x+(c_0-1)\lambda x)}{\log(c_0 \theta'x/T)}.
\end{equation*}
To obtain the corresponding inequality for $\psi$, we use \cite[Cor. 5.1]{BKLNW_21}, which states that for all $x\geq e^{40}$,
\begin{equation*}
    0<\psi(x)-\theta(x)<a_1x^{1/2}+a_2x^{1/3},
\end{equation*}
with $a_1=1+1.93378\cdot 10^{-8}$ and $a_2=1.0432$. Thus,
\begin{align*}
    \sum_{n\in I(x,u)}\Lambda(n)&\leq \psi(x_+)-\psi(x_-) +\log x\notag\\
    &<c_0ux\cdot \mathcal{E}_1(x,T)+a_1\sqrt{x_+}+a_2\sqrt[3]{x_+}+\log x\\
    &\leq c_0ux\cdot\mathcal{E}_1(x,T)+\mathcal{E}_2(x),
\end{align*}
where
\begin{equation*}
    \mathcal{E}_2(x) = a_1(x+(c_0-1)\lambda x)^{1/2} + a_2(x+(c_0-1)\lambda x)^{1/3} + \log x.
\end{equation*}
Substituting this into \eqref{perronbound2}, we obtain
\begin{align}\label{secondsumbound}
    \frac{1.785 e^{\kappa\lambda}}{T^2}\int^{\lambda}_{\theta'/T}\frac{1}{u^3}\sum_{I(x,u)}\Lambda(n)\mathrm{d}u &\leq \frac{1.785 e^{\kappa\lambda}}{T^2}\int_{\theta'/T}^{\lambda}\left(\frac{c_0 x\cdot \mathcal{E}_1(x,T)}{u^2}+\frac{\mathcal{E}_2(x)}{u^3}\right)\mathrm{d}u\notag\\
    &\leq 1.785 e^{\kappa\lambda}\left(\frac{c_0 x\cdot \mathcal{E}_1(x,T)}{\theta' T}+\frac{\mathcal{E}_2(x)}{2(\theta')^2}\right).
\end{align}
    
Hence, for all $x\geq x_M$ and some $T^* \in[T,2T]$ we have
\begin{align}\label{psi-estimate}
        &\psi(x) = x-\sum_{|\gamma|\leq T^*}\frac{x^{\rho}}{\rho}+O^*\left(K\frac{x}{T}(\log x)^{1-\omega} + 1.785\cdot\frac{ex\log x}{2\lambda^2 T^2} \right) \notag\\
        &\qquad\qquad\qquad +O^{*}\left( 1.785 e^{\kappa\lambda}\left(\frac{c_0 x\cdot \mathcal{E}_1(x,T)}{\theta' T}+\frac{\mathcal{E}_2(x)}{2(\theta')^2}\right) \right).
\end{align}
Finally, to obtain Theorem \ref{mainthm} and the values in Table \ref{maintable}, we optimise over $\lambda$ for specific $x_M$ and $\omega$, using $K$ from \cite[Table 3]{CH_DJ_Perron1}. Further computational details are given in \eqref{codeeq}.

Theorem \ref{mainthm-intervals} is proven in the same manner, using Corollary \ref{intervalperron} in place of Corollary \ref{newCor-MTPerron}.

\section{Application: primes between consecutive powers}\label{sectPowers}

In this section we prove Theorem \ref{powerthm}, which is an explicit result on primes between consecutive powers that improves \cite[Thm.~1]{cully2023primes} and \cite[Thm.~1.3]{CH_DJ_Perron1}. We use a similar method to that in \cite{cully2023primes} and \cite{CH_DJ_Perron1}, but use our new result for $\psi(x+h)-\psi(x)$ in Theorem \ref{mainthm-intervals}. We also make use of some recent improvements to zero-free regions for the Riemann zeta-function (listed in Appendix \ref{appzf}). See \eqref{codeeq} for further details about the computations used in this section.

To find primes in the interval $(n^m, (n+1)^m)$, we redefine $n = x^\frac{1}{m}$ and consider the slightly smaller interval $(x,x+h]$ with $h=mx^{1-1/m}$. Let $x_M$, $\alpha$, $\omega$, $\lambda$ and $M$ be as in Theorem \ref{mainthm-intervals} so that for $x\geq x_M$ and $\max\{51,\log^2x\}<T<(x^{\alpha}-2)/4$, there exists a $T^*\in[T,2T]$ for which
\begin{equation} \label{psipsi} 
    \psi(x+h) - \psi(x) \geq \ h - \left| \sum_{|\gamma|\leq T^*}  \frac{(x+h)^\rho-x^\rho}{\rho} \right| - M \frac{G(x,h)}{T},
\end{equation}
where $G(x,h) = (x+h)(\log(x+h))^{1-\omega} + x(\log x)^{1-\omega}$. If the right-hand side of \eqref{psipsi} is positive for some range of $x$, there must be prime powers in $(x,x+h]$ for those $x$. The number of prime powers $p^k$ with $k\geq 2$ can then be suitably bounded so that we are only left with primes.

To estimate \eqref{psipsi}, we first bound the sum over zeta zeros. The following bounds are given in the proof of \cite[Thm.~1]{cully2023primes}, where more detail is provided. We have
\begin{equation*}
    \left| \frac{(x+h)^\rho-x^\rho}{\rho} \right| \leq \int_x^{x+h} u^{\beta -1} \mathrm{d}u \leq hx^{\beta-1},
\end{equation*}
where $\beta\in(0,1)$ is the real part of $\rho$. Hence,
\begin{equation*}
    \left| \sum_{|\gamma|\leq T^*} \frac{(x+h)^\rho-x^\rho}{\rho} \right| \leq h \sum_{|\gamma|\leq 2T} x^{\beta-1}.
\end{equation*}

This sum can then be estimated by writing
\begin{align*}
    \sum_{|\gamma|\leq 2T} (x^{\beta-1}-x^{-1}) = \sum_{|\gamma| \leq 2T} \int_0^\beta x^{\sigma-1} (\log x) \mathrm{d}\sigma = \int_0^1 \sum_{\substack{|\gamma|\leq 2T \\ \beta \geq \sigma}} x^{\sigma-1}(\log x) \mathrm{d}\sigma,
\end{align*}
which re-arranges to 
\begin{equation}\label{xbetaeq}
    \sum_{|\gamma|\leq 2T} x^{\beta-1} = 2\int_0^1 x^{\sigma-1}\log x  \sum_{\substack{0<\gamma\leq 2T \\ \beta > \sigma}} 1 \ \mathrm{d}\sigma + 2x^{-1} \sum_{0<\gamma\leq 2T} 1.
\end{equation}
This expression allows us to use bounds for $N(T)$ and $N(\sigma, T)$, which respectively denote the zero-counting function and zero-density function of $\zeta(s)$. We can also use zero-free regions for $\zeta(s)$ to curtail the upper end-point of the integral. 

The most recent estimate for $N(T)$ is given by \cite[Thm.~1.1,(1.4)]{H_S_W_22}; for $N(\sigma,T)$ the most recent is in \cite[Thm.~1.1,(1.7)]{K_L_N_2018}, with updated constants in \cite[Table 3]{DJ_Y_2023}. For the zero-free region, we use a combination of the current best explicit results: see Appendix \ref{appzf}. We denote $\sigma\geq 1-\nu(T)$ as the region in which $\zeta(s)$ has no zeros, where $\nu(T)$ is defined in \eqref{mainnueq}.

The same working as in the proof of \cite[Thm.~1]{cully2023primes} follows, beginning with
\begin{align} \label{penguin}
    \sum_{|\gamma|\leq 2T} x^{\beta-1} = \frac{2N(2T)}{x}  +  \frac{2\log x}{x}\left( \int_0^{3/5} N(2T) x^{\sigma} \mathrm{d}\sigma  +  \int_{3/5}^{1-\nu(T)} N(\sigma,2T) x^{\sigma} \mathrm{d}\sigma \right).
\end{align}
Choosing to split the integral at $3/5$ is a somewhat arbitrary choice: any value in $[1/2, 1-\nu(T)]$ would work. However, after some experimentation we found the choice had a negligible impact on the final results. Using the previously mentioned estimates for $N(T)$ and $N(\sigma,T)$ gives
\begin{align} \label{N_int}
&\int_0^{3/5} N(2T) x^{\sigma} \mathrm{d}\sigma  +  \int_{3/5}^{1-\nu(2T)} N(\sigma,2T) x^{\sigma} \mathrm{d}\sigma \leq \frac{2T\log(2T)}{2\pi}  \left( \frac{x^{3/5}-1}{\log x} \right)  \\
& \quad  + C_1 (2T)^{8/3}\log^5(2T)  \left( \frac{ W^{1-\nu(2T)}-W^{\sigma_1}}{\log W}\right)  + C_2 \log^2(2T) \left( \frac{x^{1-\nu(2T)}-x^{3/5}}{\log x} \right),   \nonumber 
\end{align}
where $C_1=C_1(1)$ and $C_2=C_2(3/5)$ are defined in \cite[Thm.~1.1]{K_L_N_2018}, and 
\begin{equation*}
    W=x ((2T)^{\frac{4}{3}} \log (2T))^{-2}.
\end{equation*} 
Note that the specific values of $C_1$ and $C_2$ are due to $C_1(\sigma)$ being an increasing function of $\sigma$ and $C_2(\sigma)$ being decreasing: see equations (4.72) and (4.73) in \cite{K_L_N_2018}.

Next, we choose $T=x^\mu/2$ for any $\mu\in(0,1)$ that satisfies $\max\{51,\log^2 x\}\leq T\leq (x^\alpha -2)/4$. The factor of $1/2$ in the choice of $T$ is to account for the sum in \eqref{xbetaeq} being over $|\gamma|\leq 2T$. We have, with $W=x^{1-\frac{8\mu}{3}}/(\mu\log x)^2$,
\begin{align} \label{sum_fns}
\sum_{|\gamma|\leq 2T} x^{\beta-1} & < F(x) = \frac{\mu\log x}{\pi x^{2/5-\mu}} +  \frac{2C_1\mu^3 \left(W^{-\nu(x^\mu)} - W^{-2/5} \right)\log^4 x}{\log W}  +   \frac{2C_2\mu^2\log^2 x}{x^{\nu(x^\mu)}}.
\end{align}
Using this bound in \eqref{psipsi}, we reach
\begin{align*}
    \psi(x+h) - \psi(x) \geq h - h F(x) - 2M \frac{G(x,h)}{x^\mu}.
\end{align*}
This can readily be converted to an inequality for the prime counting function $\theta(x)=\sum_{p\leq x}\log p$. In particular, by \cite[Corollary 5.1]{BKLNW_21} and the fact that ${\psi(x+h)>\theta(x+h)}$, we have
\begin{equation*}
    \theta(x+h)-\theta(x)\geq h - h F(x) - 2M \frac{G(x,h)}{x^\mu}-E(x),
\end{equation*}
where $E(x)=1.001(x^{1/2}+x^{1/3})$ provided $\log x\geq 100$. Therefore, there are primes in $(x,x+h]$ for all $x\geq x_M$ satisfying 
\begin{equation} \label{condition_psi} 
    1 - F(x) - 2M \frac{G(x,h)}{x^\mu h} -\frac{E(x)}{h} > 0.
\end{equation}
This condition depends on: the choice of power $m$ in $h=mx^{1-1/m}$, the smallest value in the range $x_M$, and the parameters $\mu$, $\omega$, and $\alpha$.

It remains to find the smallest $m$ for which \eqref{condition_psi} holds for all $x\geq x_0>1$. This involves optimising over $\mu$, and is a similar process to that in \cite[Sect.~4,pg.~114]{cully2023primes}. For any $\mu > 1/m$, the left-hand side of \eqref{condition_psi} increases to 1 for sufficiently large $x$. The largest-order term is in $F(x)$, and is determined by the zero-free region. If \eqref{condition_psi} is true at the point where the zero-free region switches from using Lemma \ref{andrew_zf} to Lemma \ref{MTY-vk-region}, then \eqref{condition_psi} will hold for all $x$ beyond this. There may also be a range of $\mu$ for which this is true, so the best choice of $\mu$ will keep \eqref{condition_psi} positive for the widest range of $x$ below this zero-free region crossover point.

We follow the previous steps for $m=90$, as this appears to be the smallest $m$ we can achieve with the given constants. From Table \ref{maintable} we can use $M = 6.555$ over $\log x \geq 10^3$, and $M=5.602$ over $\log x \geq 4\cdot 10^3$, each with $\omega = 0.9$. Numerically optimising over $\mu$, we find that for $\mu = 0.0113$, \eqref{condition_psi} is true for $10^3\leq x\leq 4\cdot 10^3$, and with $\mu=0.0112$ \eqref{condition_psi} holds for $x\geq 4\cdot 10^3$.

There are a few remaining steps to reach Theorem \ref{powerthm}. Firstly, the values for $M$ come from using $\alpha = 1/85$ in Theorem~\ref{mainthm}, which restricts $T<(x^{1/85}-2)/4$. Given that we set $T=x^\mu/2$, this upper bound on $T$ implies that for a given $\mu$, the result will only be valid for sufficiently large $x$. For $\mu=0.0113$, this condition implies that Theorem \ref{mainthm} only holds for $\log x\geq 1492$. Fortunately, we can supplement this result with the interval estimates for primes in \cite{CH_L_corr}. The results in \cite{CH_L_corr} for $x\geq 4\cdot 10^{18}$ and $x\geq e^{600}$ verify that there are primes between consecutive $90^{\text{th}}$ powers for $\log(4\cdot 10^{18}) \leq \log x \leq 2767$. This is found by solving for when the interval estimates in \cite{CH_L_corr} sit within the consecutive powers interval. Lastly, the computations of \cite{O_H_P_14} verify the consecutive $90^{\text{th}}$ powers interval for the remaining $x\geq 1$. This proves Theorem \ref{powerthm}.

\section*{Acknowledgements}
Thanks to our supervisor, Tim Trudgian for his ideas and support when writing this paper. We also thank Ethan Lee for some useful comments, and Olivier Ramar\'e for his correspondence regarding his earlier paper.

\newpage
\appendix
\section{Current zero-free regions}\label{appzf}

We use a zero-free region composed of the following four results.

\begin{lemma}{\cite[Thm.~1.3]{mossinghoff2024explicit}}\label{MTYclassical}
    For $|t|\geq 2$ there are no zeros of $\zeta(\beta+it)$ in the region $\beta\geq 1-\nu_1(t)$, where
    \begin{equation*}
        \nu_1(t)=\frac{1}{5.558691\log |t|}.
    \end{equation*}
\end{lemma}

\begin{lemma}{\cite[Thm.~3]{Ford_2002}}\label{fordclassical}
    For $|t|\geq 3$ there are no zeros of $\zeta(\beta+it)$ in the region $\beta\geq 1-\nu_2(t)$, where
    \begin{equation*}
       \nu_2(t)=\frac{1}{R(|t|)\log |t|},
    \end{equation*}
    with
    \begin{equation*}
        R(t)=\frac{J(t)+0.685+0.155\log\log t}{\log t\left(0.04962-\frac{0.0196}{J(t)+1.15}\right)},
    \end{equation*}
    and
    \begin{equation*}
        J(t)=\frac{1}{6}\log t+\log\log t+\log(0.618).
    \end{equation*}
\end{lemma}
\begin{remark}
    The expression for $J(t)$ in Lemma \ref{fordclassical} has been updated from \cite{Ford_2002} by using a more recent bound \cite[Theorem 1.1]{hiary2024improved}.
\end{remark}

\begin{lemma}{\cite[Cor.~1.2]{yang2024explicit}}\label{andrew_zf}
    For $|t|\geq 3$ there are no zeros of $\zeta(\beta+it)$ in the region $\beta\geq 1-\nu_3(t)$, where
    \begin{equation*}
       \nu_3(t)=\frac{\log\log|t|}{21.233\log |t|}.
    \end{equation*}
\end{lemma}

\begin{lemma}{\cite[Thm.~1.2]{bellotti2024explicit}}\label{MTY-vk-region}
    For $|t|\geq 3$ there are no zeros of $\zeta(\beta+it)$ in the region $\beta\geq 1-\nu_4(t)$ where
    \begin{equation}\label{vkregion}
        \nu_4(t)=\frac{1}{53.989\log^{2/3}|t|(\log\log |t|)^{1/3}}.
    \end{equation}
\end{lemma}

To get the widest zero-free region at any given $T$ we set 
\begin{equation}\label{mainnueq}
    \nu(T)=
    \begin{cases}
        \frac{1}{2},&\text{if $|T|\leq 3\cdot 10^{12}$},\\
        \max\{\nu_1(T),\nu_2(T),\nu_3(T),\nu_4(T)\}&\text{otherwise,}
    \end{cases}
\end{equation}
where the constant $3\cdot 10^{12}$ comes from Platt and Trudgian's computational verification of the Riemann hypothesis \cite{P_T-RH_21}. Note that $\nu_2(t)\geq\nu_1(t)$ for $t\geq\exp(46.3)$, $\nu_3(t)\geq\nu_2(t)$ for $t\geq\exp(170.3)$ and $\nu_4(t)\geq\nu_3(t)$ for $t\geq\exp(482036)$.

\newpage

\section{Updating bounds for an integral over $\zeta'/\zeta$}\label{appk}
In \cite[Section 4]{CH_DJ_Perron1} the following theorem is proven.
\begin{theorem}[{\cite[Theorem 4.1]{CH_DJ_Perron1}}]\label{wolkeprop}
    Let $\alpha\in(0,1]$ and $\omega\in[0,1]$. There exists constants $K$ and $x_K$ such that if $x\geq x_K$ and $\max\{51,\log x\}<T<\frac{x^{\alpha}}{2}-1$, 
    \begin{equation}\label{wolkeeq}
        \frac{1}{2\pi i}\int_{1+\varepsilon-iT}^{1+\varepsilon+iT}\left(-\frac{\zeta'}{\zeta}(s)\right) \frac{x^s}{s}\mathrm{d}s = x-\sum_{\substack{\rho=\beta+i\gamma\\|\gamma|\leq T}}\frac{x^\rho}{\rho}+O^*\left(\frac{Kx}{T}(\log x)^{1-\omega}\right),
    \end{equation}
    where $\varepsilon=1/\log x$. 
\end{theorem}
Here, corresponding values of $x_K$, $\alpha$, $\omega$ and $K$ are given in \cite[Table 3]{CH_DJ_Perron1}. A key component in computing values of $K$ is the use of explicit zero-free regions. As there have recently been improvements to explicit zero-free regions (Appendix A), the values in \cite[Table 3]{CH_DJ_Perron1} can be updated. We define $\nu(T)$ to be as in \eqref{mainnueq} and then use this updated zero-free region in the bounds (4.19) and (4.20) in \cite{CH_DJ_Perron1}. This gives us the updated table below (Table \ref{ktable}), with values for $x_K$, $\alpha$, $\omega$ and $K$ that can be used in Theorem \ref{wolkeprop}. Further computational details are given in \eqref{codeeq}.

\def\arraystretch{1.5}
\begin{table}[h]
\centering
\caption{Some corresponding values of $x_K$, $\alpha$, $\omega$ and $K$ for Theorem~\ref{wolkeprop} that updates \cite[Table 3]{CH_DJ_Perron1}. Note that $\overline{\omega}$ and $D$ are parameters one optimises over in the original calculation in \cite{CH_DJ_Perron1}.}

\begin{tabular}{|c|c|c|c|c|c|}
\hline
$\log(x_K)$ & $\alpha$ & $\omega$ & $\overline{\omega}$ & $D$ & $K$\\
\hline
$40$ & $1/2$ & $0$ & $0$ & --- & $2.033$\\
\hline
$10^3$ & $1/2$ & $0$ & $0$ & --- & $1.645$\\
\hline
$10^{10}$ & $1/2$ & $0.3$ & $0.3$ & $0.54$ & $2.274$\\
\hline
$10^{13}$ & $1/2$ & $1$ & $1.4$ & $0.50$ & $0.6367$\\
\hline
$10^3$ & $1/10$ & $0.2$ & $0.2$ & $0.35$ & $1.215$\\
\hline
$10^{10}$ & $1/10$ & $0.9$ & $0.9$ & $0.53$ & $3.425$\\
\hline
$10^3$ & $1/100$ & $0.8$ & $0.8$ & $0.46$ & $1.027$\\
\hline
$10^{10}$ & $1/100$ & $1$ & $1.5$ & $0.50$ & $0.6367$\\
\hline
$10^{3}$ & $1/85$ & $0.9$ & $0.9$ & $0.50$ & $1.218$\\
\hline
$4\cdot 10^{3}$ & $1/85$ & $0.9$ & $0.9$ & $0.50$ & $1.176$\\
\hline
\end{tabular}
\label{ktable}
\end{table}

We remark that the lower bound on $T$ in Theorem \ref{wolkeprop} is slightly stronger than the bound we desire for Theorem \ref{mainthm}. In fact, one could weaken the bounds on $T$ in Theorem \ref{wolkeprop} to get marginally better values for $K$. However, we chose to keep this lower bound unchanged as to maintain consistency between this paper and \cite{CH_DJ_Perron1}.

\newpage

\printbibliography

\end{document}